\documentclass[12pt]{amsart}
\usepackage{amsfonts, amssymb, latexsym, xcolor, graphicx}
\usepackage{algpseudocode,algorithm,mathtools, hyperref, tikz-cd}

\newtheorem{theorem}{Theorem}[section]
\newtheorem{proposition}[theorem]{Proposition}
\newtheorem{lemma}[theorem]{Lemma}

\newtheorem*{claim*}{Claim}
\newtheorem{corollary}[theorem]{Corollary}
\newtheorem{Main Conjecture}[theorem]{Main Conjecture}

\theoremstyle{definition}
\newtheorem{definition}[theorem]{Definition}
\newtheorem{example}[theorem]{Example}

\theoremstyle{remark}

\newtheorem{remark}[theorem]{Remark}

\theoremstyle{plain}

\newcommand{\cellsize}{19}
\newlength{\cellsz} 
\newsavebox{\cell}
\sbox{\cell}{\begin{picture}(\cellsize,\cellsize)
\put(0,0){\line(1,0){\cellsize}}
\put(0,0){\line(0,1){\cellsize}}
\put(\cellsize,0){\line(0,1){\cellsize}}
\put(0,\cellsize){\line(1,0){\cellsize}}
\end{picture}}
\newcommand\cellify[1]{\def\thearg{#1}\def\nothing{}%
\ifx\thearg\nothing
\vrule width0pt height\cellsz depth0pt\else
\hbox to 0pt{\usebox{\cell} \hss}\fi%
\vbox to \cellsz{
\vss
\hbox to \cellsz{\hss$#1$\hss}
\vss}}
\newcommand\tableau[1]{\vtop{\let\\\cr
\baselineskip -16000pt \lineskiplimit 16000pt \lineskip 0pt
\ialign{&\cellify{##}\cr#1\crcr}}}

\DeclareMathOperator{\linear}{Lin}
\DeclareMathOperator{\decomp}{Dec}

\newcommand{\excise}[1]{}%{$\star$\textsc{#1}$\star$}

\begin{document}
\pagestyle{plain}
\title{Presenting the cohomology of a Schubert variety:\\ Proof of the minimality conjecture}

\author{Avery St.~Dizier}
\address{Dept.~of Mathematics, Michigan State University, East Lansing, MI 48824, USA} 
\email{stdizier@msu.edu}

\author{Alexander Yong}
\address{Dept.~of Mathematics, U.~Illinois at Urbana-Champaign, Urbana, IL 61801, USA} 
\email{ayong@illinois.edu}

\date{August 31, 2023}

\begin{abstract} 
A minimal presentation of the cohomology ring of the flag manifold $GL_n/B$ was given in [A.~Borel, 1953]. This presentation was extended 
by [E.~Akyildiz--A.~Lascoux--P.~Pragacz, 1992] to a non-minimal one for all Schubert varieties. Work of [V.~Gasharov--V.~Reiner, 2002] gave a short, \emph{i.e.} polynomial-size,
presentation for a subclass of Schubert varieties that includes the smooth ones. In [V.~Reiner--A.~Woo--A.~Yong, 2011], a general shortening was found; it implies an exponential upper bound of $2^n$ on the number of generators required. That work states
a minimality
conjecture whose significance would be an exponential lower bound of $\frac{\sqrt{2}^{n+2}}{\sqrt{\pi n}}$ on the number
of generators needed in worst case, giving the first obstructions to short presentations. We prove the minimality conjecture. Our proof uses the Hopf algebra structure of the ring of symmetric functions.
\end{abstract}

\maketitle

\section{Introduction}

\subsection{Background} Let $X=Fl_n({\mathbb C})$ be the \emph{complete flag manifold}; its points are complete flags of subspaces of ${\mathbb C}^n$,
\[F_{\bullet}=\{\langle \vec 0\rangle\subset F_1\subset F_2\subset \cdots\subset F_{n-1}\subset {\mathbb C}^n\},\]
where $F_i$ is a $i$-dimensional linear subspace of ${\mathbb C}^n$. 
In 1953, A.~Borel \cite{Borel} gave a presentation of its integral cohomology ring:
\[H^{*}(X)\cong {\mathbb{Z}}[x_1,\ldots,x_n]/I^{S_n},\]
where $I^{S_n}$ is the ideal generated by symmetric polynomials of positive degree. A feature of Borel's presentation is its \emph{shortness}: the ideal $I^{S_n}$ is generated by the elementary symmetric polynomials
$e_{i}(x_1,\ldots,x_n)$, $1\leq i\leq n$. This exhibits that $H^*(X)$ is a complete intersection, with $n$ generators and $n$ relations.

Let $GL_n$ be the Lie group of invertible $n\times n$ matrices and $B$ its Borel subgroup of invertible upper triangular matrices. $GL_n$ acts transitively on $X$ 
whereas $B$ is the stabilizer of the standard basis flag given by $F_i:={\rm Span} \ \{\vec e_1, \vec e_2,\ldots, \vec e_i\}$. By the orbit-stabilizer theorem, $X$ may be
identified (topologically) with $GL_n/B$. The finitely many $B$-orbits of $X$ are called the \emph{Schubert cells} $X_w^{\circ}$ and are indexed by permutations $w$ in the symmetric group $S_n$ of permutations of $\{1,2,\ldots,n\}$. Each cell is isomorphic to the affine space
${\mathbb C}^{\ell(w)}$. Together, they form a CW-decomposition for $X$ where each cell has even real dimension.
Their closures, the \emph{Schubert varieties} 
\[X_w\coloneqq \overline{X_w^{\circ}},\] 
provide a ${\mathbb Z}$-basis of the integral homology $H_*(X)$ and their
Poincar\'e duals 
$\sigma_w=[X_w]^*$ 
give a ${\mathbb Z}$-basis for the integral cohomology ring $H^*(X)$.

The \emph{Bruhat decomposition} is
\[X_w=\coprod_{u\leq w} X_u^{\circ},\]
where $u\leq w$ refers to \emph{(strong) Bruhat order} on $S_n$. Thus, $X_w$ inherits a CW-decomposition from $Fl_n$. The map on cohomology
\[H^*(X)\to H^*(X_w)\]
that is induced by the inclusion of $X_w$ in $Fl_n$ is a surjection with kernel 
\begin{equation}
\label{eqn:apriorigens}
I_w\coloneqq \mathrm{Span}_\mathbb{Z} \left\{\sigma_u\mid u\not\leq w\right\};
\end{equation}
see \cite{RWY}. Therefore one obtains a Borel-type presentation
\[H^*(X_w)\cong H^*(X)/I_w
.\]

The list of generators for $I_w$ given in (\ref{eqn:apriorigens}) is quite redundant in general since they span $I_w$ \emph{linearly}. What are more efficient lists of generators for $I_w$?

In 1992, E.~Akyildiz--A.~Lascoux--P.~Pragacz \cite{ALP} made the first step towards minimizing the generators of $I_w$. To state their result, define a \emph{grassmannian} permutation to be $u\in S_n$ with a unique descent, 
\emph{i.e.}, a position $k$ such that $u(k)>u(k+1)$.

\begin{theorem}[{\cite[Theorem~2.2]{ALP}}]\label{thm:ALP}
For any $w\in S_n$, the ideal $I_w$ defining $H^*(X_w)$ as a quotient of $H^*(X)$ is generated
by the cohomology classes $\sigma_u$ where $u\not\leq w$ and $u$ is grassmannian.
\end{theorem}

In 2002, V.~Gasharov--V.~Reiner \cite{Gasharov.Reiner} showed that when $X_w$
is \emph{defined by inclusions} (a family of Schubert varieties that include the smooth ones) then $I_w$ can be generated by $n^2$ many generators. In addition, for the subclass of \emph{Ding's Schubert varieties} \cite{Ding1,Ding2}, they offered an even smaller generating set consisting of $n$ generators; an application is
given by M.~Develin--J.~Martin--V.~Reiner \cite{DMR}.

Can one always give such a ``short'' presentation of
$H^*(X_w)$? Formally, one asks:

\begin{center}
\emph{Is it always possible to generate $I_w$ by $O({\sf poly}(n))$ many generators?}
\end{center}

In 2011, V.~Reiner, A.~Woo, and the second author \cite{RWY} further refined Theorem~\ref{thm:ALP}. Complementing this result was a minimality conjecture which, if true, implies a 
negative answer to the above question. That is, there is a family of Schubert varieties $X_{w^{(n)}}$ for which $I_{w^{(n)}}$ requires exponentially many generators.

\subsection{The main result} 
The goal of this paper is to prove the minimality conjecture of \cite{RWY}.

A useful way to think about Bruhat order was introduced by A.~Lascoux and M.-P.~Schutzenberger \cite{LS}.
They show there is a ``base'' $B\subset S_n$ which is minimal with respect to set-theoretic inclusion such that the
map 
\[\varphi: (S_n,\leq ) \to (2^B,\subseteq), \ u\mapsto \{b\in B| b\leq u\}\]
is a poset isomorphism of $(S_n,\leq )$ with its image. In fact, $B$ consists of the \emph{bigrassmannian} permutations, namely, those $v\in S_n$ such that both $v$ and $v^{-1}$ are grassmannian. 

The bigrassmannian permutations (besides the identity) are indexed by integers $r,s,t$ such that
$1\leq t\leq r,s\leq n$ and $t>r+s-n$. Let $v_{r,s,t,n}\in S_n$ be the bigrassmannian permutation uniquely
characterized by having unique descent at $r$, $v^{-1}_{r,s,t,n}$ having descent at $s$ and $v_{r,s,t,n}(t)=s+1$. 
In one-line notation, one explicitly has
\begin{multline}\nonumber
	v_{r,s,t,n}=1,2,\ldots,t-1,\ s+1,s+2,\ldots,s+r-t+1,\\
	t,t+1,\ldots,s, \
	s+r-t+2,s+r-t+3,\ldots,n;
\end{multline}
see \cite[Lemma~4.1]{RWY}.

For each basic element of $S_n$, a bigrassmannian permutation
$v$, define the \emph{basic ideal} of $H^*(X)$ as
\[J_v\coloneqq \mathrm{Span}_\mathbb{Z}\left\{\sigma_u\mid u\geq v\right\}.\]
It is shown in \cite{RWY} that $I_w$ decomposes into basic ideals. Specifically,
\begin{equation}
\label{eqn:thedecompabc}
I_w=\sum_{v\in {\mathcal E}(w)} J_v,
\end{equation}
where ${\mathcal E}(w)$ is the set of $u \in S_n$ which are minimal in the Bruhat order among those not
below $w$. The set ${\mathcal E}(w)$ is referred to as the \emph{essential set}, and consists only of bigrassmannian permutations \cite[Theorem 1.1]{RWY}.

As explained in \cite[Section~3]{RWY}, the projection 
\[X\to \mathrm{Gr}_{r,n}\]
onto the Grassmannian of $r$-planes in $\mathbb{C}^n$
induces a split-inclusion $H^*(\mathrm{Gr}_{r,n})\xhookrightarrow{\iota} H^*(X)$. Consequently (see \cite[Proposition 3.1]{RWY}), elements of  $H^*(\mathrm{Gr}_{r,n})$ generate $J_v$ in $H^*(X)$ if and only if they generate the contraction $J_v\cap H^*(\mathrm{Gr}_{r,n})$ inside $H^*(\mathrm{Gr}_{r,n})$.

In particular, elements of $H^*(\mathrm{Gr}_{r,n})$ \emph{minimally} generate $J_v$ in $H^*(X)$ if and only if they \emph{minimally} generate $J_v\cap H^*(\mathrm{Gr}_{r,n})$ in $H^*(\mathrm{Gr}_{r,n})$. This allows one to study the generation of $J_v$ inside $H^*(\mathrm{Gr}_{r,n})$ instead of $H^*(X)$ without loss of generality, enabling the use of symmetric function theory.

%It is established in \cite[Section~3]{RWY} (see specifically Proposition~3.1 and Theorem~3.2) that, without loss, the ideals $J_v$ of $H^*(X)$ can be considered instead in the cohomology ring $H^*(\mathrm{Gr}_{r,n})$ of the Grassmannian of $r$-planes in $\mathbb{C}^n$. This allows one to study $J_v$ by working with symmetric functions. 

Let $\Lambda=\Lambda(x_1,x_2,\ldots)$ be the ring of symmetric functions with integral coefficients. This space has a ${\mathbb Z}$-linear basis consisting of \emph{Schur functions} $s_{\lambda}$,
where $\lambda$ is an integer partition. Consider the ideal 
\[I_{r,n-r} = \mathrm{Span}_\mathbb{Z}\left\{s_\lambda\mid \lambda\nsubseteq (n-r)^r \right\} \]
of $\Lambda$. Then one has the following standard presentation of the cohomology of Grassmannians:
\[H^*(\mathrm{Gr}_{r,n})\cong \Lambda/I_{r,n-r}. \]

We will use an overline to denote the projection $\Lambda\to\Lambda/I_{r,n-r}$. Specifically, we will write $\overline{\Lambda}$ as a shorthand for $\Lambda/I_{r,n-r}$. We will also distinguish between $s_{\lambda}\in \Lambda$, the usual Schur function, and $\overline{s}_{\lambda}$ for its image in $\overline{\Lambda}$. 

We resolve the minimality conjecture \cite[Conjecture~4.10]{RWY} about the basic ideals $J_v$:

\begin{theorem}\label{thm:main}
Given a bigrassmannian permutation $v=v_{r,s,t,n}\in S_n$ let
\begin{align*}
i&\coloneqq s-t+1,\\
j&\coloneqq r-t+1,\\
a&\coloneqq \min(n-r-i,r-j),\\
b&\coloneqq \min(i,j).
\end{align*}
Then the basic ideal $J_v$, thought of as an ideal of $H^*(\mathrm{Gr}_{r,n})\cong \Lambda/I_{r,n-r}$, is minimally generated by
\begin{equation}
\label{eqn:minimallist}
\{\overline{s}_{\mu} \mid i^j\subseteq \mu \subseteq (i^j,b^a)\}.
\end{equation}
\end{theorem}

That (\ref{eqn:minimallist}) generates $J_v$ is \cite[Theorem 4.8]{RWY}. The conjectural part, which we prove,
is the minimality claim, that is, no generator can be removed without changing the ideal. 

As explained in \cite[Section~4.4]{RWY}, the minimality from Theorem~\ref{thm:main}
implies obstructions to short presentations of $H^*(X_w)$. First, let 
$w^{(n)}:=w_{r,s,t,n}$ be defined as in \cite[Corollary~4.5]{RWY}; explicitly, by \cite[Section~8]{Reading},
\begin{multline} 
w_{r,s,t,n}=n,n-1,\ldots,(n-r+t+1), s,s-1,\ldots,s-t+1,\\
n-r+t,n-r+t-3,\ldots,s+1,s-t,s-t-1,\ldots,1.
\end{multline}
Then
\begin{equation}
\label{eqn:I=J}
I_{w^{(n)}}=J_{v_{r,s,t,n}}.
\end{equation}
Second, let $n=4m$, $r=s=2m$, and $t=m+1$ for some positive integer $m$. Thus $r=n-r=2m$, $i=j=m$, and $a=b=m$. 
Therefore, Theorem~\ref{thm:main} combined with (\ref{eqn:I=J}) implies $I_{w^{(n)}}$ requires an exponentially growing number
\[{2m\choose m}\sim \frac{4^{m}}{\sqrt{\pi m}}=\frac{(\sqrt{2})^{n+2}}{\sqrt{\pi n}}\] 
of generators. This proves the first exponential lower bound to accompany the $2^n$ upper bound
established in \cite{RWY}.

\subsection{Organization} The remainder of this paper contains our proof of Theorem~\ref{thm:main}. 

Our proof strategy is as follows. If the generators (\ref{eqn:minimallist}) are not minimal, there is a 
syzygy of the form (\ref{eqn:goal-syzygy}). In Section~\ref{sec:theequations}, using the famous Littlewood--Richardson rule, this can be rephrased as a nontrivial solution to a system of linear equations. We divide the system into two subfamilies, ``tall'' and ``wide'' equations (Section~\ref{sec:theequations}). All of these
equations are homogeneous except for some tall equations. 

The key step is to show that coefficient-wise, the tall equations are implied by the wide equations. We prove this by 
transforming the problem in terms of tensors of Schur functions and making use of the Hopf algebra structure of symmetric
functions; this is achieved in Section~\ref{sec:Hopf}. 
As we explain in Section~\ref{sec:conclusion}, this implies the aforementioned system of linear equations is
inconsistent, thus ruling out the existence of a syzygy (\ref{eqn:goal-syzygy}) and completing the proof.

\section{Setting up the linear equations}\label{sec:theequations}

We briefly recall the basic properties of Schur functions and Littlewood--Richardson coefficients. See, \emph{e.g.}, \cite{fultonyt, ECII} for thorough introductions. For a partition $\lambda$ (which we identify with its Young diagram in English notation), the \emph{Schur function} $s_\lambda=s_\lambda(x_1,x_2,\ldots)$ is the generating function for semistandard Young tableaux of shape $\lambda$. Schur functions form a $\mathbb{Z}$-basis of $\Lambda$. The \emph{Littlewood--Richardson coefficients} $c_{\lambda,\mu}^{\nu}$ are the structure constants for the Schur basis:
\[s_\lambda s_{\mu} = \sum_{\nu} c_{\lambda,\mu}^{\nu} s_{\nu}. \]

The \emph{Littlewood--Richardson rule} gives an explicit combinatorial description of $c_{\lambda\mu}^{\nu}$
(the version we use is Theorem~\ref{thm:pictures}).
A fact that we rely on throughout is that Schur functions are homogeneous of degree $|\lambda|=\lambda_1+\lambda_2+\cdots$. In particular, $c_{\lambda,\mu}^{\nu}$ can only be nonzero if $|\nu|=|\lambda|+|\mu|$. Similarly, it follows from the Littlewood--Richardson rule that $c_{\lambda,\mu}^\nu$ can only be nonzero if $\lambda,\mu\subseteq \nu$ (identifying partitions with their Young diagrams in English notation).

For the $\overline{s}_\lambda$ in $\overline{\Lambda}$, the multiplication is the same except partitions $\rho$ must lie in the box $(n-r)^r$ in order for $\overline{s}_\rho$ to be nonzero:
\[\overline{s}_\lambda\overline{s}_{\mu} = \sum_{\nu\subseteq (n-r)^r} c_{\lambda,\mu}^{\nu} \overline{s}_{\nu}. \]

Now, we return to addressing Theorem \ref{thm:main}. Fix $v_{r,s,t,n}\in S_n$, and assign $i,j,a,b$ as in Theorem \ref{thm:main}. To prove the minimality of (\ref{eqn:minimallist}), we must show that for any $\rho\subseteq b^a$, the ring $\overline{\Lambda}$ does not admit syzygies of the form
\begin{equation}
	\label{eqn:goal-syzygy}
	\overline{s}_{(i^j,\rho)}=\sum_{\substack{\lambda\subseteq b^a\\\lambda\neq \rho}}g_{\lambda}\overline{s}_{(i^j,\lambda)}.
\end{equation}
There is one important technicality we dispense with immediately: if $i=n-r$, then $a=0$ and Theorem \ref{thm:main} is trivial. Thus we will assume that $i<n-r$, so $a\geq 1$.

Actually, we will prove the slightly stronger statement that no expression (\ref{eqn:goal-syzygy}) exists for any $n,r,i,j,a,b\geq 1$ with
\[r<n,\quad a+j\leq r,\quad a+i\leq n-r, \quad \mbox{and}\quad b\leq i,j. \]

Fix any such integers $(n,r,i,j,a,b)$. Suppose one can find $\rho\subseteq b^a$ and polynomials $g_{\lambda}\in \overline{\Lambda}$ such that (\ref{eqn:goal-syzygy}) holds. 

Since the Schur functions are homogeneous, both $\Lambda$ and $I_{r,n-r}$ are graded by degree. The quotient $\overline{\Lambda}$ inherits this grading. Thus, we can compare the degree $N=|\rho|$ homogeneous components of both sides of (\ref{eqn:goal-syzygy}) and rearrange to obtain a homogeneous syzygy of the form
\begin{equation}
	\label{eqn:syzygy}
	\sum_{\substack{\lambda\subseteq b^a\\|\lambda|=N}}\overline{s}_{(i^j,\lambda)}f_{\lambda}=\sum_{\substack{\lambda\subseteq b^a\\|\lambda|<N }}f_{\lambda}\overline{s}_{(i^j,\lambda)} \quad\mbox{with $f_{\lambda}$ homogeneous and }\deg f_{\lambda} = N- |\lambda|.
\end{equation}

By homogeneity, the left-hand side of (\ref{eqn:syzygy}) is simply a $\mathbb{Z}$-linear combination (that is, each $f_{\lambda}\in {\mathbb Z}$). We ignore the left-hand side of (\ref{eqn:syzygy}) for now and instead focus on the right-hand side.
Since $f_{\lambda}\in \overline{\Lambda}$, we can expand in the Schur basis to obtain
\[f_{\lambda} = \sum_{\substack{\theta\subseteq (n-r)^r\\|\theta|=N-|\lambda|}}A_{\lambda,\theta}\overline{s}_{\theta}. \]
Then 
\begin{align*}
	\sum_{\substack{\lambda\subseteq b^a\\|\lambda|<N }}f_{\lambda}\overline{s}_{(i^j,\lambda)}
	 &=\sum_{\substack{\lambda\subseteq b^a\\|\lambda|<N }}\sum_{\substack{\theta\subseteq (n-r)^r\\|\theta|=N-|\lambda|}}A_{\lambda,\theta}\overline{s}_{(i^j,\lambda)}\overline{s}_\theta\\
	 &=\sum_{\substack{\lambda\subseteq b^a\\|\lambda|<N }}\sum_{\substack{\theta\subseteq (n-r)^r\\|\theta|=N-|\lambda|}}A_{\lambda,\theta}\sum_{\nu\subseteq (n-r)^r} c_{(i^j,\lambda),\theta}^{\nu}\, \overline{s}_{\nu}\\
	 &=\sum_{\nu\subseteq (n-r)^r}\biggr( \sum_{\substack{\lambda\subseteq b^a\\|\lambda|<N }}\sum_{\substack{\theta\subseteq (n-r)^r\\|\theta|=N-|\lambda|}} c_{(i^j,\lambda),\theta}^{\nu} \, A_{\lambda,\theta}\biggr) \overline{s}_{\nu}.
	 \end{align*}

Viewing the parameters $A_{\lambda,\theta}$ as indeterminates, a syzygy of the form (\ref{eqn:syzygy}) for some $n,r,i,j,a,b,N$ is equivalent to a solution of the simultaneous linear equations
\begin{equation}
	\label{eqn:eqs}
	\biggr\{ \sum_{\substack{\lambda\subseteq b^a\\|\lambda|<N }}\sum_{\substack{\theta\subseteq (n-r)^r\\|\theta|=N-|\lambda|}} c_{(i^j,\lambda),\theta}^{\nu} \, A_{\lambda,\theta}=\chi(\nu)\,\biggr| \,\nu\subseteq (n-r)^r \biggr\},
\end{equation}
where $\chi(\nu)=f_\rho$ when $\nu=(i^j,\rho)$ for some $\rho\subseteq b^a$ with $|\rho|=N$, and $\chi(\nu)=0$ otherwise.

The following definition records the assumptions on the parameters used in the preceding considerations.
\begin{definition}
	Call an integer tuple $\varphi=(n,r,i,j,a,b,N)$ \emph{valid} if $n,r,i,j,a,b,N\geq 1$ with
	\[r<n,\quad a+j\leq r,\quad a+i\leq n-r,\quad b\leq i,j,\quad\mbox{and}\quad N\leq ab. \]
\end{definition}

Due to vanishing of Littlewood--Richardson coefficients, not all partitions $\nu\subseteq (n-r)^r$ contribute meaningfully to (\ref{eqn:eqs}).
\begin{definition}
	Call a partition $\nu$ with $i^j\subseteq\nu\subseteq (n-r)^r$ and $|\nu|=ij+N$ \emph{allowable}.
\end{definition}
When $\nu$ is not allowable, both $\chi(\nu)$ and all the Littlewood--Richardson coefficients $c_{(i^j,\lambda),\theta}^{\nu}$ vanish. In this case both sides of the corresponding equation in (\ref{eqn:eqs}) are zero, so we can assume that $\nu$ is always allowable in~(\ref{eqn:eqs}). We will actually show that the equations (\ref{eqn:eqs}) are inconsistent by focusing on a smaller subset of them.

\begin{definition}
	Call an allowable partition $\nu$ \emph{decomposable} if $\nu_{j+1}\leq b$ (taking $\nu_{j+1}=0$ if $\ell(\nu)= j$), and $\ell(\nu)\leq j+a$. Write $\decomp_{\varphi}$ for the set of decomposable partitions $\nu$. We will decompose any $\nu\in\decomp_{\varphi}$ into 3 smaller partitions via Figure \ref{fig:decomposition}.
\end{definition}
We will similarly say that an equation in (\ref{eqn:eqs}) is \emph{decomposable} if the partition $\nu$ it corresponds to is decomposable. 

\begin{figure}[ht]
	\begin{center}
		\includegraphics{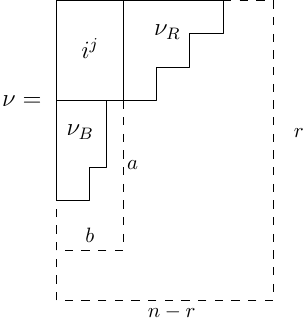}
	\end{center}
	\caption{}
	\label{fig:decomposition}
\end{figure}

We divide the decomposable partitions into two groups.
\begin{definition}
	We say a decomposable partition $\nu$ is \emph{tall} if $\nu_1=i$ (so $\nu_R$ is empty), and \emph{wide} otherwise.
\end{definition}
Similarly, refer to the equation of a tall (resp.\ wide) partition $\nu$ as tall (resp.\ wide).

\begin{lemma}
	\label{lem:implication-of-syzygy}
	Suppose for some valid $\varphi=(n,r,i,j,a,b,N)$ there is a homogeneous syzygy of degree $N$ in $\overline{\Lambda}$ of the form 
	\[\sum_{\substack{\lambda\subseteq b^a\\|\lambda|=N}}\overline{s}_{(i^j,\lambda)}f_{\lambda}=\sum_{\substack{\lambda\subseteq b^a\\|\lambda|<N }}f_{\lambda}\overline{s}_{(i^j,\lambda)}\]
	with the left-hand side nonzero. 
	Then the decomposable equations 
	\begin{equation}
		\label{eqn:decomposable}
		\biggr\{ \sum_{\substack{\lambda\subseteq b^a\\|\lambda|<N }}\sum_{\substack{\theta\subseteq (n-r)^r\\|\theta|=N-|\lambda|}} c_{(i^j,\lambda),\theta}^{\nu} \, A_{\lambda,\theta}=\chi(\nu)\,\biggr| \,\nu\in \decomp_{\varphi} \biggr\}
	\end{equation}
	are consistent, where $\chi(\nu)\in\mathbb{Z}$ is defined for $\nu\in \decomp_{\varphi}$ by 
	\[\chi(\nu)=\begin{cases}
		f_{\nu_B} &\mbox{ if $\nu$ is tall,}\\
		0 &\mbox{ if $\nu$ is wide.}
	\end{cases} \]
\end{lemma}
\begin{proof}
	The above considerations show that the lemma is true if we replace (\ref{eqn:decomposable}) with the superset (\ref{eqn:eqs}). Clearly, removing equations from a consistent linear system yields a consistent linear system.
\end{proof}

We show that the decomposable equations are inconsistent in Theorem \ref{thm:inconsistent}. We will use the following notation for left-hand sides of the individual equations in (\ref{eqn:decomposable}).
\begin{definition}
	Given a valid $\varphi=(n,r,i,j,a,b,N)$ and any $\nu\in \decomp_{\varphi}$, define
	\[\linear_{\varphi}(\nu)= \sum_{\substack{\lambda\subseteq b^a\\|\lambda|<N }}\sum_{\substack{\theta\subseteq (n-r)^r\\|\theta|=N-|\lambda|}} c_{(i^j,\lambda),\theta}^{\nu} \, A_{\lambda,\theta}. \]
\end{definition}

\begin{example}
	\label{exp:equations}
	Let $n=12$, $r=6$, $i=j=a=b=3$, and $N=4$. In this case, the allowable partitions $\nu$ are:
	\begin{center}
		\includegraphics[scale=1]{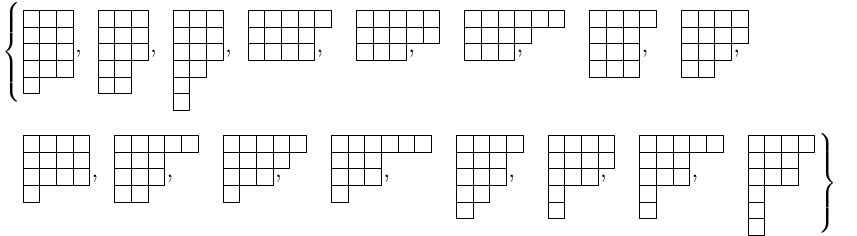}
	\end{center}
	These partitions are all decomposable. We record their corresponding linear combinations in Figure~\ref{fig:m3n4equations}.
\end{example}

\begin{example}
	If $n=17$, $r=9$, $i=5$, $j=3$, $a=4$, $b=2$, and $N=10$, then not all allowable partitions are decomposable.
	This is the case for the two partitions shown below.
	\begin{center}
		\includegraphics[scale=.8]{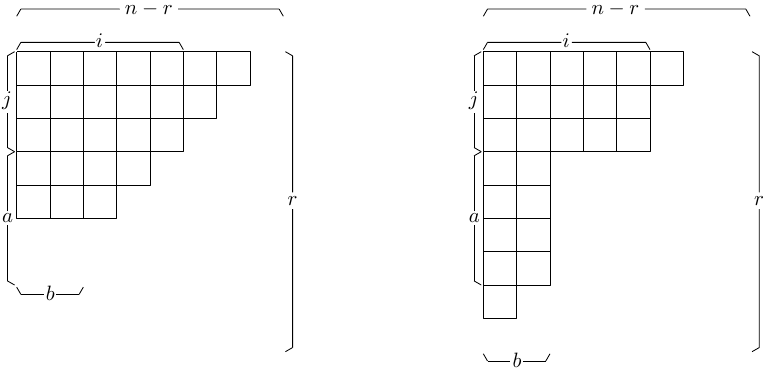}
	\end{center}
	In particular, for these parameter values there are strictly more (nontrivial) equations in (\ref{eqn:eqs}) than in (\ref{eqn:decomposable}). The point of Lemma~\ref{lem:implication-of-syzygy} is that we do not need to use these equations to show the system (\ref{eqn:eqs}) is inconsistent.
\end{example}

\begin{figure}[ht]
	\begin{center}
		\includegraphics[scale=1.15]{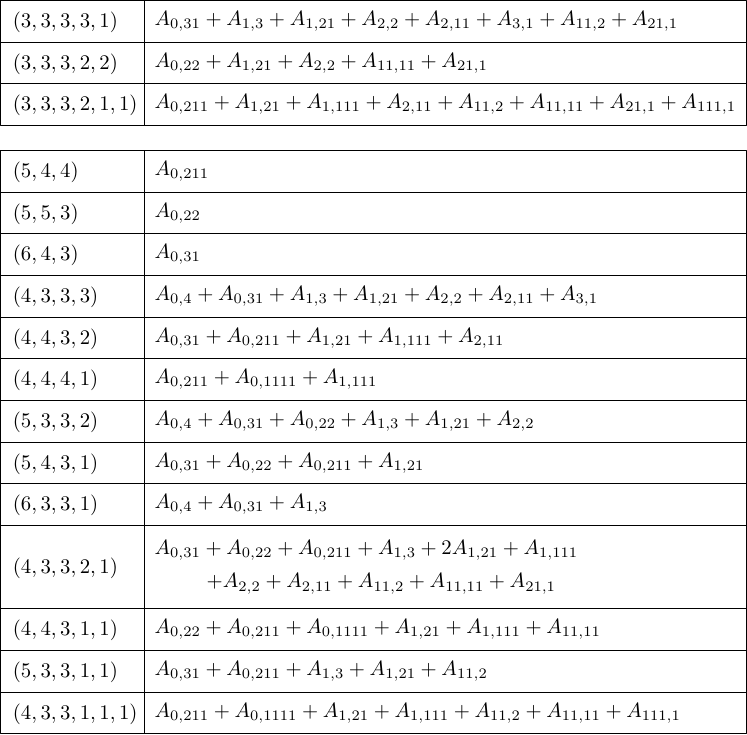}
		\caption{The decomposable equations for Example \ref{exp:equations}, divided into tall and wide equations.}
		\label{fig:m3n4equations}
	\end{center}
\end{figure}

\section{The Hopf Algebra of symmetric functions}\label{sec:Hopf}
Fix a valid tuple $\varphi=(n,r,i,j,a,b,N)$. View the linear forms $\linear_{\varphi}(\nu)$ as vectors in the ambient vector space $V$ with basis 
$\left\{A_{\lambda,\theta} \right\}_{\lambda,\theta}$, where $\lambda,\theta\subseteq (n-r)^r$. 
Let $T$ be the linear isomorphism $T:V\to \overline{\Lambda}\otimes \overline{\Lambda}$ defined by
\[T(A_{\lambda,\theta}) = \overline{s}_{\lambda}\otimes \overline{s}_{\theta}.\]
We use this isomorphism to give a useful perspective on the vectors $\linear_{\varphi}(\nu)\in V$ via 
the (ring) product structure on $\overline{\Lambda}\otimes\overline{\Lambda}$ and 
the coproduct map
\[\Delta:\Lambda\to \Lambda\otimes\Lambda.\]

Recall the coproduct $\Delta$ on symmetric functions acts on Schur functions by
\[\Delta(s_{\nu}) = \sum_{\lambda,\mu} c_{\lambda,\mu}^\nu (s_\lambda\otimes s_{\mu}).\]
Since the projection $\Lambda\to\overline{\Lambda}$ naturally induces a projection 
\[\Lambda\otimes \Lambda \to \overline{\Lambda}\otimes \overline{\Lambda}, \]
we will also denote the latter with overlines.

\begin{definition}
	\label{def:cp}
	We define a linear endomorphism $\mathrm{CP}:\overline{\Lambda}\otimes \overline{\Lambda}\to \overline{\Lambda}\otimes \overline{\Lambda}$ as follows.
	For a simple tensor $\overline{s}_\lambda\otimes \overline{s}_\mu\in \overline{\Lambda}\otimes\overline{\Lambda}$, define
	\begin{align}
		\mathrm{CP}(\overline{s}_\lambda\otimes \overline{s}_\mu) = \begin{cases}
			\overline{\Delta(s_{\lambda})} - (\overline{s}_{\lambda}\otimes 1) &\mbox{ if $\overline{s}_\mu= 1$},\\
			\overline{\Delta(s_{\lambda})}(1\otimes \overline{s}_{\mu}) &\mbox{ if $\overline{s}_\mu\neq 1$}.
		\end{cases}
	\end{align}
	Define $\mathrm{CP}$ on all of $\overline{\Lambda}\otimes \overline{\Lambda}$ by extending linearly.
\end{definition}

\begin{example}
	\label{exp:cp-example-calculations}
	For example with $n=12$ and $r=6$ (as in Example \ref{exp:equations}), one can compute 
	\begin{align*}
		\mathrm{CP}(\overline{s}_{21}\otimes \overline{s}_1) &= \overline{\Delta(s_{21})}(1\otimes \overline{s}_1)\\
		&=\left(1\otimes \overline{s}_{21}
			+\overline{s}_1\otimes \overline{s}_{11}
			+\overline{s}_1\otimes \overline{s}_2
			+\overline{s}_2\otimes \overline{s}_1 
			+ \overline{s}_{11}\otimes \overline{s}_1
			+\overline{s}_{21}\otimes 1
		  \right)(1\otimes \overline{s}_1)\\
		&=(1\otimes \overline{s}_{21}\overline{s}_1)
			+(\overline{s}_1\otimes \overline{s}_{11}\overline{s}_1)
			+(\overline{s}_1\otimes \overline{s}_2\overline{s}_1)
			+(\overline{s}_2\otimes \overline{s}_1^2)
			+ (\overline{s}_{11}\otimes \overline{s}_1^2)
			+(\overline{s}_{21}\otimes \overline{s}_1),
	\end{align*}
	and 
	\begin{align*}
		\mathrm{CP}(\overline{s}_{22}\otimes 1) &=\overline{\Delta(s_{22})}- (\overline{s}_{22}\otimes 1)\\
		&=1\otimes \overline{s}_{22}
		+\overline{s}_1\otimes \overline{s}_{21}
		+\overline{s}_2\otimes \overline{s}_2
		+\overline{s}_{11}\otimes \overline{s}_{11} 
		+ \overline{s}_{21}\otimes \overline{s}_1.
	\end{align*}
\end{example}

We think of each decomposable $\nu$ as corresponding to the simple tensor $\overline{s}_{\nu_B}\otimes \overline{s}_{\nu_R}$. Let us consider how the map $\mathrm{CP}$ acts on this tensor when $\nu$ is tall versus when $\nu$ is wide.

When $\nu$ is tall, $|\nu_B|=N$ and $\overline{s}_{\nu_R}=\overline{s}_{(0)}=1$. Then by the first case of Definition \ref{def:cp},
\begin{align*}
	\mathrm{CP}(\overline{s}_{\nu_B}\otimes 1) &= \overline{\Delta(s_{\nu_B})} - (\overline{s}_{\nu_B}\otimes 1)\\
	&=\sum_{\lambda\subsetneq \nu_B}\sum_{\theta} c_{\lambda,\theta}^{\nu_B} (\overline{s}_\lambda\otimes \overline{s}_{\theta})\\	&=\sum_{\substack{\lambda\subseteq \nu_B\\|\lambda|<N}}\sum_{\substack{\theta\subseteq (n-r)^r\\|\theta|=N-|\lambda|}} c_{\lambda,\theta}^{\nu_B} (\overline{s}_\lambda\otimes \overline{s}_{\theta}).
\end{align*}
When $\nu$ is wide, we have $|\nu_B|<N$ so $|\nu_R|>0$. Then the second case of Definition \ref{def:cp} implies 
\begin{align*}
	\mathrm{CP}(\overline{s}_{\nu_B}\otimes \overline{s}_{\nu_R}) &= \overline{\Delta(s_{\nu_B})}(1\otimes \overline{s}_{\nu_R})\\
	&=\sum_{\lambda\subseteq \nu_B}\sum_{\mu\subseteq \nu_B} c_{\lambda,\mu}^{\nu_B} \overline{s}_\lambda \otimes (\overline{s}_\mu \overline{s}_{\nu_R}) \\
	&=\sum_{\substack{\lambda\subseteq \nu_B\\|\lambda|<N}}\sum_{\mu\subseteq \nu_B}\sum_{\theta\subseteq (n-r)^r}c_{\lambda,\mu}^{\nu_B} \, c_{\mu,\nu_R}^{\theta} (\overline{s}_{\lambda} \otimes \overline{s}_{\theta})\\
	&=\sum_{\substack{\lambda\subseteq \nu_B\\|\lambda|<N}}\sum_{\substack{\theta\subseteq (n-r)^r\\|\theta|=N-|\lambda|}}\left(\sum_{\mu\subseteq \nu_B}c_{\lambda,\mu}^{\nu_B} \, c_{\mu,\nu_R}^{\theta}\right) (\overline{s}_{\lambda} \otimes \overline{s}_{\theta}).
\end{align*}

Recall $V$ denotes the vector space with basis $\left\{A_{\lambda,\theta} \mid \lambda,\theta\subseteq (n-r)^r\right\}$, and $T$ is the linear isomorphism $T:V\to \overline{\Lambda}\otimes \overline{\Lambda}$ defined by $T(A_{\lambda,\theta}) = \overline{s}_{\lambda}\otimes \overline{s}_{\theta}$.
\begin{proposition}
	\label{prop:eq-to-tensors}
	For any decomposable partition $\nu$,
	\[T(\linear_{\varphi}(\nu)) = \mathrm{CP}(\overline{s}_{\nu_B}\otimes \overline{s}_{\nu_R}).\]
\end{proposition}

To prove Proposition \ref{prop:eq-to-tensors}, we first review Littlewood--Richardson pictures, a combinatorial model of Littlewood--Richardson coefficients due to James and Peel \cite{James.Peel}, and Zelevinsky \cite{Zelev}.

A \emph{skew diagram} is obtained from two partitions $\nu$ and $\lambda$ with 
$\lambda\subseteq \nu$ by aligning the northwest box of each and taking the set difference. This diagram is
denoted $\nu/\lambda$. Any partition $\theta$ is also a skew diagram $\theta=\theta/(0)$.

A \emph{picture} between two skew diagrams is a bijection between their boxes such that if a box $A$ is weakly above and weakly left of a box $B$ in either diagram, then the corresponding boxes $A'$ and $B'$ of the other diagram appear in order in the reverse row numbering (the numbering of boxes right-to-left in each row, working top-to-bottom). The following result can be found in \cite[Chapter 5.3]{fultonyt}.

\begin{theorem}
\label{thm:pictures}
	The Littlewood--Richardson coefficient $c_{\lambda,\mu}^\nu$ equals the number of pictures between $\nu/\lambda$ and $\mu$.
\end{theorem}

\begin{example}
	The Littlewood--Richardson coefficient $c_{3222,431}^{543221}=4$ is witnessed by the pictures in Figure \ref{fig:pictures}.
\end{example}
\begin{figure}[h]
	\begin{center}
		\includegraphics[scale=.75]{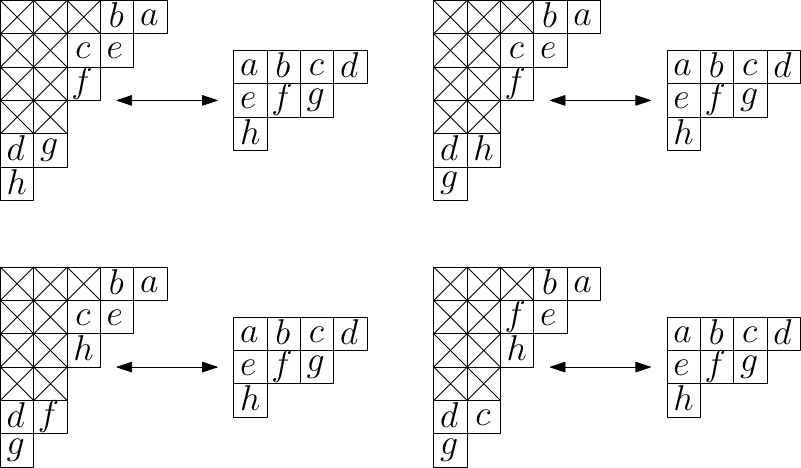}
	\end{center}
	\caption{Example of Theorem~\ref{thm:pictures}.}
	\label{fig:pictures}
\end{figure}

\begin{proof}[Proof of Proposition \ref{prop:eq-to-tensors}.]
	We must show that $T(\linear_{\varphi}(\nu)) = \mathrm{CP}(\overline{s}_{\nu_B}\otimes \overline{s}_{\nu_R})$, that is
	\[\mathrm{CP}(\overline{s}_{\nu_B}\otimes \overline{s}_{\nu_R}) = \sum_{\substack{\lambda\subseteq \nu_B\\|\lambda|<N}}\sum_{\substack{\theta\subseteq (n-r)^r\\|\theta|=N-|\lambda|}} c_{(i^j,\lambda),\theta}^\nu\, (\overline{s}_{\lambda}\otimes \overline{s}_\theta). \]
	Suppose first that $\nu$ is tall. Then as we computed after Example \ref{exp:cp-example-calculations},
	\[
		\mathrm{CP}(\overline{s}_{\nu_B}\otimes 1) =\sum_{\substack{\lambda\subseteq \nu_B\\|\lambda|<N}}\sum_{\substack{\theta\subseteq (n-r)^r\\|\theta|=N-|\lambda|}} c_{\lambda,\theta}^{\nu_B} (\overline{s}_\lambda\otimes \overline{s}_{\theta}).
	\]
	Thus it suffices to prove that
	\[ c_{(i^j,\lambda),\theta}^\nu = c_{\lambda,\theta}^{\nu_B}\]
	for all $\lambda,\theta\subseteq (n-r)^r$ with $\lambda\subsetneq \nu_B$ and $|\theta|=N-|\lambda|$. This is immediate from Theorem \ref{thm:pictures}, since the skew shapes $\nu_B/\lambda$ and $\nu/(i^j,\lambda)$ are (essentially) the same.
	
	Now, assume that $\nu$ is wide. Again consulting the computations following Example \ref{exp:cp-example-calculations}, we have 
	\[
		\mathrm{CP}(\overline{s}_{\nu_B}\otimes \overline{s}_{\nu_R}) 
		=\sum_{\substack{\lambda\subseteq \nu_B\\|\lambda|<N}}\sum_{\substack{\theta\subseteq (n-r)^r\\|\theta|=N-|\lambda|}}\left(\sum_{\mu\subseteq \nu_B}c_{\lambda,\mu}^{\nu_B} \, c_{\mu,\nu_R}^{\theta}\right) (\overline{s}_{\lambda} \otimes \overline{s}_{\theta}).
	\]
	Hence it suffices to show that
	\[c_{(i^j,\lambda),\theta}^\nu=\sum_{\mu\subseteq \nu_B}c_{\lambda,\mu}^{\nu_B} \, c_{\mu,\nu_R}^{\theta}\]
	for all $\lambda,\theta\subseteq (n-r)^r$ with $\lambda\subseteq \nu_B$, $|\lambda|<N$, and $|\theta|=N-|\lambda|$.
	
	Denote the conjugate of a partition $\tau$ by $\tau'$. The map $\omega:\Lambda\to\Lambda$ defined by $\omega(s_\rho)=s_{\rho'}$ for each partition $\rho$ is known to be ring involution (see for instance \cite[Chapter 6.2]{fultonyt}). This implies that $c_{\alpha,\beta}^{\gamma}=c_{\alpha',\beta'}^{\gamma'}$ for any $\alpha,\beta,\gamma$. Thus, the proof is complete if we can show 
	\[c^{\nu'}_{(i^j,\lambda)',\theta'} = \sum_{\mu\subseteq \nu_B} c_{\lambda',\mu'}^{(\nu_B)'}\, c_{\mu',(\nu_R)'}^{\theta'}. \]
	
	It straightforward to check that a correspondence of skew diagrams
	\[\nu'/(i^j,\lambda)' \longleftrightarrow \theta' \]
	is a picture if and only if it induces pictures 
	\begin{align*}
		(\nu_B)'/\lambda' &\longleftrightarrow \mu'\subseteq \theta' \mbox{ for some $\mu$, and }\\
		\theta'/\mu'&\longleftrightarrow (\nu_R)'.\qedhere
	\end{align*}
\end{proof}

We aim to show that the linear system (\ref{eqn:decomposable}) of Lemma \ref{lem:implication-of-syzygy} is always inconsistent, yielding a contradiction. The main idea is to show that for each tall $\nu$, one can write $\linear_{\varphi}(\nu)$ as a linear combination of the vectors $\{\linear_{\varphi}(\rho) \mid \rho \mbox{ is wide}\}$. We will see in Lemma \ref{lem:tall-dependent-on-wide} that Algorithm~1 below accomplishes this. 

\begin{algorithm}
	\caption{}
	\label{alg}
	\begin{algorithmic}
		\State input valid $\varphi=(n,r,i,j,a,b,N)$.
		\State input a tall $\nu\in\decomp_{\varphi}$.		
		\\
		\State initialize the tensor
		\[\xi^{(0)}\coloneqq \mathrm{CP}(\overline{s}_{\nu_B}\otimes 1) =\sum_{\lambda\subsetneq \nu_B}\sum_{\substack{\theta\subseteq (n-r)^r\\|\theta|=N-|\lambda|}} c_{\lambda,\theta}^{\nu_B} (\overline{s}_\lambda\otimes \overline{s}_{\theta}). \]
		\State initialize $m=1$.
		
		\While{$\xi^{(m-1)}$ contains a nonzero term of the form
		\[\gamma(\overline{s}_{\lambda}\otimes (\overline{s}_{\mu^{(1)}}\cdots \overline{s}_{\mu^{(k)}})) \mbox{ with $|\lambda|>0$,} \mbox{\quad(where $k\geq 1$, and $\gamma\in\mathbb{Z}$)} \]
		}
		\State Without expanding the product on the right factor of each tensor, set
		\begin{align*}
			\xi^{(m)}&\coloneqq\xi^{(m-1)}-\gamma\mathrm{CP}(\overline{s}_{\lambda}\otimes (\overline{s}_{\mu^{(1)}}\cdots \overline{s}_{\mu^{(k)}}))\\
			&=\xi^{(m-1)}-\gamma\overline{\Delta(s_{\lambda})} (1\otimes (\overline{s}_{\mu^{(1)}}\cdots \overline{s}_{\mu^{(k)}}))\\
			&=\xi^{(m-1)}-\sum_{\rho,\tau \subseteq (n-r)^r} \gamma \,c_{\rho,\tau}^{\lambda} (\overline{s}_\rho\otimes (\overline{s}_{\mu^{(1)}}\cdots \overline{s}_{\mu^{(k)}}\overline{s}_\tau )).
		\end{align*}
		\State increment $m$.
		\EndWhile
		
		\State \Return $\xi^{(m-1)}$
	\end{algorithmic}
\end{algorithm}
We first demonstrate this algorithm, then proceed with analyzing it.
\begin{example}
	\label{exp:coproduct-process}
	Continuing Example \ref{exp:equations}, $\nu=(3,3,3,3,1)$ is decomposable and tall, with
	\[ \linear_{\varphi}(\nu) = A_{0,31}+A_{1,3}+A_{1,21}+A_{2,2}+A_{2,11}+A_{3,1}+A_{11,2}+A_{21,1}.\]

	The corresponding tensor is
	\begin{align*}
		\mathrm{CP}(\overline{s}_{31}\otimes 1)&=\overline{\Delta(s_{31})}-\overline{s}_{31}\otimes 1\\
		&=1\otimes \overline{s}_{31} + \overline{s}_{1}\otimes \overline{s}_{3} + \overline{s}_{1}\otimes \overline{s}_{21} + \overline{s}_{2}\otimes \overline{s}_{2} +
		\overline{s}_{2}\otimes \overline{s}_{11}\\
		&\phantom{blah}+
		\overline{s}_{3}\otimes \overline{s}_{1} +
		\overline{s}_{11}\otimes \overline{s}_{2} + 
		\overline{s}_{21}\otimes \overline{s}_{1}.
	\end{align*}
	This tensor is $\xi^{(0)}$. It contains the term $\overline{s}_{1}\otimes \overline{s}_{3}$, so we eliminate it by setting
	\begin{align*}
		\xi^{(1)} &= \xi^{(0)} - \mathrm{CP}(\overline{s}_{1}\otimes \overline{s}_{3}) \\
		&= \xi^{(0)} - \overline{\Delta(s_{1})}(1\otimes \overline{s}_{3}) \\
		&= \xi^{(0)} - (1\otimes \overline{s}_{1} + \overline{s}_{1}\otimes 1 )(1\otimes \overline{s}_{3})\\
		&= \xi^{(0)} - (1\otimes \overline{s}_{1}\overline{s}_{3} + \overline{s}_{1}\otimes \overline{s}_{3} )\\
		&=1\otimes \overline{s}_{31} - 1\otimes \overline{s}_{1}\overline{s}_{3} + \overline{s}_{1}\otimes \overline{s}_{21} + \overline{s}_{2}\otimes \overline{s}_{2} +
		\overline{s}_{2}\otimes \overline{s}_{11}\\
		&\phantom{blah}+
		\overline{s}_{3}\otimes \overline{s}_{1} +
		\overline{s}_{11}\otimes \overline{s}_{2} + 
		\overline{s}_{21}\otimes \overline{s}_{1}.
	\end{align*}
	For the second iteration, use the term $\overline{s}_{2}\otimes \overline{s}_{2}$ occurring in $\xi^{(1)}$. We eliminate it with:
	\begin{align*}
		\xi^{(2)} &= \xi^{(1)} - \mathrm{CP}(\overline{s}_{2}\otimes \overline{s}_{2}) \\
		&= \xi^{(1)} - \overline{\Delta(s_{2})}(1\otimes \overline{s}_{2}) 
		\end{align*}
		\begin{align*}
		&= \xi^{(1)} - (1\otimes \overline{s}_2 + \overline{s}_1\otimes \overline{s}_1+ \overline{s}_2\otimes 1)(1\otimes \overline{s}_2) \\
		&= \xi^{(1)} - (1\otimes \overline{s}_2^2 + \overline{s}_1\otimes \overline{s}_1\overline{s}_2+ \overline{s}_2\otimes \overline{s}_2) \\
		&=1\otimes \overline{s}_{31} 
		- 1\otimes \overline{s}_{1}\overline{s}_{3} 
		- 1\otimes \overline{s}_2^2
		+ \overline{s}_{1}\otimes \overline{s}_{21} 
		- \overline{s}_1\otimes \overline{s}_1\overline{s}_2
		+\overline{s}_{2}\otimes \overline{s}_{11}\\
		&\phantom{blah}+
		\overline{s}_{3}\otimes \overline{s}_{1} +
		\overline{s}_{11}\otimes \overline{s}_{2} + 
		\overline{s}_{21}\otimes \overline{s}_{1}.
	\end{align*}
	In the third iteration, we eliminate the term $-\overline{s}_{1}\otimes \overline{s}_{1}\overline{s}_{2}$ occurring in $\xi^{(2)}$. Then
	\begin{align*}
		\xi^{(3)} &= \xi^{(2)} + \mathrm{CP}(\overline{s}_{1}\otimes \overline{s}_{1}\overline{s}_{2}) \\
		&= \xi^{(2)} + \overline{\Delta(s_{1})}(1\otimes \overline{s}_{1}\overline{s}_{2}) \\
		&= \xi^{(2)} + (1\otimes \overline{s}_1 + \overline{s}_1\otimes 1)(1\otimes \overline{s}_1\overline{s}_2) \\
		&= \xi^{(2)} + (1\otimes \overline{s}_1^2\overline{s}_2 + \overline{s}_1\otimes \overline{s}_1\overline{s}_2)\\
		&=1\otimes \overline{s}_{31} 
		- 1\otimes \overline{s}_{1}\overline{s}_{3} 
		- 1\otimes \overline{s}_2^2
		+ 1\otimes \overline{s}_1^2\overline{s}_2
		+ \overline{s}_{1}\otimes \overline{s}_{21} 
		+\overline{s}_{2}\otimes \overline{s}_{11}\\
		&\phantom{blah}+
		\overline{s}_{3}\otimes \overline{s}_{1} +
		\overline{s}_{11}\otimes \overline{s}_{2} + 
		\overline{s}_{21}\otimes \overline{s}_{1}.
	\end{align*}
	
	Continuing in this fashion, one eventually arrives at
	\begin{align*}
		&1\otimes \overline{s}_{31} 
		- 2(1\otimes \overline{s}_{1}\overline{s}_3) 
		- 2(1\otimes \overline{s}_1\overline{s}_{21}) 
		- (1\otimes \overline{s}_2^2)
		-2(1\otimes \overline{s}_{2}\overline{s}_{11})\\
		&\phantom{blah}
		+6(1\otimes \overline{s}_{1}^2\overline{s}_2)
		+3(1\otimes \overline{s}_1^2\overline{s}_{11})
		-3(1\otimes \overline{s}_1^4)
	\end{align*}
	at which point the algorithm terminates.
\end{example}

Now, we analyze this algorithm in general.
\begin{proposition}
	\label{clm:1}
	Algorithm~1 terminates after a finite number of steps. 
\end{proposition}
\begin{proof}
	To prove this claim, consider an iteration
	\begin{align*}
		\xi^{(m)}&=\xi^{(m-1)}-\gamma\mathrm{CP}(\overline{s}_{\lambda}\otimes (\overline{s}_{\mu^{(1)}}\cdots \overline{s}_{\mu^{(k)}}))\\
		&=\xi^{(m-1)}-\gamma \sum_{\rho,\tau \subseteq (n-r)^r} c_{\rho,\tau}^{\lambda} (\overline{s}_\rho\otimes (\overline{s}_{\mu^{(1)}}\cdots \overline{s}_{\mu^{(k)}}\overline{s}_\tau )).
	\end{align*}
	All tensors appearing with nonzero coefficient in the sum have left tensor factor $\overline{s}_{\rho}$ with $|\rho|<|\lambda|$ except for $\overline{s}_{\lambda}\otimes (\overline{s}_{\mu^{(1)}}\cdots \overline{s}_{\mu^{(k)}})$, which cancels with the identical term in $\xi^{(m-1)}$. Therefore the number of tensors $s_{\alpha}\otimes \bullet$ appearing in $\xi^{(m)}$ with $|\alpha|\geq |\lambda|$
	is strictly smaller than the number in $\xi^{(m-1)}$. 
	
	Hence for $i$ sufficiently large, the number of tensors $s_{\alpha}\otimes \bullet$ appearing in $\xi^{(m)}$ with $|\alpha|>1$ will be zero, at which point the algorithm terminates.
\end{proof}

\begin{proposition}
	\label{clm:2}
	The terms 
	\[\overline{s}_{\lambda}\otimes (\overline{s}_{\mu^{(1)}}\cdots \overline{s}_{\mu^{(k)}})\]
	occurring with nonzero coefficient in any $\xi^{(m)}$ have integral coefficients with sign $(-1)^{k+1}$.
\end{proposition}
\begin{proof}
	That the coefficients are integral holds since they arise from repeated sums, differences, and products of Littlewood--Richardson coefficients, which are integers. By the construction of Algorithm 1, the signs start positive in $\xi^{(0)}$ and change each time a new factor is added to the right of the tensor in an update $\xi^{(m-1)}\to \xi^{(m)}$. 
\end{proof}

\begin{proposition}
	\label{clm:3}
	The output of Algorithm~1 does not depend on the choice of term in each iteration.
\end{proposition}

\begin{proof}It follows from Proposition~\ref{clm:2} that during the algorithm, the only cancellation that occurs in an update
\begin{equation}
\label{eqn:theupdate}
\xi^{(m)}=\xi^{(m-1)}-\gamma\mathrm{CP}(\overline{s}_{\lambda}\otimes (\overline{s}_{\mu^{(1)}}\cdots \overline{s}_{\mu^{(k)}}))
\end{equation}
is the cancellation of the term $\gamma\overline{s}_{\lambda}\otimes (\overline{s}_{\mu^{(1)}}\cdots \overline{s}_{\mu^{(k)}})$ in $\xi^{(m-1)}$, and its negative in $\gamma\mathrm{CP}(\overline{s}_{\lambda}\otimes (\overline{s}_{\mu^{(1)}}\cdots \overline{s}_{\mu^{(k)}}))$. Moreover, the algorithm only can
eliminate a term $\gamma\overline{s}_{\lambda}\otimes (\overline{s}_{\mu^{(1)}}\cdots \overline{s}_{\mu^{(k)}})$ with $|\lambda|>1$ by the update step
(\ref{eqn:theupdate}). Therefore the order in which we compute the update steps (\ref{eqn:theupdate}) does not effect the output.\end{proof}

\begin{definition}
	Given a valid tuple $\varphi$ and a tall partition $\nu\in\decomp_{\varphi}$, let $\mathrm{Reduce}_{\varphi}(\nu)\in\overline{\Lambda}\otimes \overline{\Lambda}$ denote the output of Algorithm~1 .
\end{definition}
\begin{example}
	Continuing Example \ref{exp:coproduct-process} with $\nu=(3,3,3,3,1)$, 
	\begin{align*}
		\mathrm{Reduce}_{\varphi}(\nu)&=
		1\otimes \overline{s}_{31} 
		- 2(1\otimes \overline{s}_{1}\overline{s}_3) 
		- 2(1\otimes \overline{s}_1\overline{s}_{21}) 
		- (1\otimes \overline{s}_2^2)
		-2(1\otimes \overline{s}_{2}\overline{s}_{11})\\
		&\phantom{blah}+6(1\otimes \overline{s}_{1}^2\overline{s}_2)
		+3(1\otimes \overline{s}_1^2\overline{s}_{11})
		-3(1\otimes \overline{s}_1^4).
	\end{align*}
	Simplifying, we obtain
	\begin{align*}
		\mathrm{Reduce}_{\varphi}(\nu) &= 1\otimes 
		\left(
		\overline{s}_{31}
		- 2 \overline{s}_{1}\overline{s}_3 
		- 2 \overline{s}_1\overline{s}_{21} 
		-  \overline{s}_2^2
		-2 \overline{s}_{2}\overline{s}_{11}
		+6 \overline{s}_{1}^2\overline{s}_2
		+3\overline{s}_1^2\overline{s}_{11}
		-3\overline{s}_1^4
		\right)\\
		&=-(1\otimes \overline{s}_{211}).
	\end{align*}
	Notice $211$ is the conjugate of $\nu_B=31$.
\end{example}

Denote the conjugate of a partition $\tau$ by $\tau'$. In general, we have:
\begin{proposition}
	\label{prop:hopf-identity-application}
	Let $\nu\in \decomp_{\varphi}$ be tall with $\nu_B\neq \emptyset$.
	Then 
	\[\mathrm{Reduce}_{\varphi}(\nu) =
		(-1)^{|\nu_B|+1}(1\otimes \overline{s}_{(\nu_B)'}). \]
\end{proposition}
\begin{proof}
	We proceed by induction on $N=|\nu_B|$. The base case is $|\nu_B|=1$. In this case, Algorithm~1 outputs $\xi^{(0)}=\mathrm{CP}(\overline{s}_1\otimes 1)$, so 
	\[\mathrm{Reduce}_{\varphi}(\nu)=\xi^{(0)}=\mathrm{CP}(\overline{s}_1\otimes 1) = \sum_{|\theta|=1} c_{(0),\theta}^{(1)} (1\otimes \overline{s}_{\theta}) = 1\otimes \overline{s}_1.\]
	Suppose $|\nu_B|=k>1$. Construct the tensor $\xi^{(0)}$ from $\nu$, so
	\begin{align*}
		\xi^{(0)}&=\mathrm{CP}(\overline{s}_{\nu_B}\otimes 1)\\ 
		 	 	 &=\sum_{\lambda\subsetneq \nu_B}\sum_{\substack{\theta\subseteq (n-r)^r\\|\theta|=N-|\lambda|}} c_{\lambda,\theta}^{\nu_B} (\overline{s}_\lambda\otimes \overline{s}_{\theta}).
	\end{align*}
	Since there are finitely many nonzero terms occurring in the double sum, let us list them out explicitly as
	\[
		\xi^{(0)}= (1\otimes \overline{s}_{\nu_B})+\sum_{p=1}^L c_{\lambda^p,\theta^p}^{\nu_B}(\overline{s}_{\lambda^p}\otimes \overline{s}_{\theta^p}),
	\]
	where $1\leq |\lambda^p|<k$ for each $p$. Set $\varphi^{p} = (n,r,i,j,a,b,|\lambda^p|)$ for each $p$. 
	From Proposition \ref{clm:3}, it follows that $\mathrm{Reduce}_{\varphi}$ satisfies the ``depth-first search'' recurrence
	\[\mathrm{Reduce}_{\varphi}(\nu) = (1\otimes \overline{s}_{\nu_B}) - \sum_{p=1}^L c_{\lambda^p,\theta^p}^{\nu_B} \mathrm{Reduce}_{\varphi}(i^j,\lambda^p)(1\otimes \overline{s}_{\theta^p}).\]
	By the induction assumption,
	\[\mathrm{Reduce}_{\varphi^{p}}(i^j,\lambda^p)(1\otimes \overline{s}_{\theta^p}) = (-1)^{|\lambda^p|+1}(1\otimes \overline{s}_{(\lambda^p)'}\overline{s}_{\theta^p}) \text{ for each }1\leq p\leq L. \]
	Thus,
	\begin{align*}
		\mathrm{Reduce}_{\varphi}(\nu) &= (1\otimes \overline{s}_{\nu_B}) + \sum_{p=1}^L c_{\lambda^p,\theta^p}^{\nu_B} (-1)^{|\lambda^p|}(1\otimes \overline{s}_{(\lambda^p)'}\overline{s}_{\theta^p})\\
		&=(1\otimes \overline{s}_{\nu_B}) +\sum_{\substack{|\lambda|\geq 1\\ \lambda\subsetneq \nu_B}}\sum_{\substack{\theta\subseteq (n-r)^r\\|\theta|=N-|\lambda|}} c_{\lambda,\theta}^{\nu_B} (-1)^{|\lambda|}(1\otimes \overline{s}_{\lambda'}\overline{s}_{\theta})\\
		&=\sum_{\lambda\subsetneq \nu_B}\sum_{\substack{\theta\subseteq (n-r)^r\\|\theta|=N-|\lambda|}} c_{\lambda,\theta}^{\nu_B} (-1)^{|\lambda|}(1\otimes \overline{s}_{\lambda'}\overline{s}_{\theta}).
	\end{align*}	
	The lemma then follows if one can prove that
	\[\sum_{\lambda,\theta}(-1)^{|\lambda|}c_{\lambda,\theta}^{\nu_B}\overline{s}_{\lambda'}\overline{s}_{\theta}=0.\]
	
	In Lemma \ref{lem:hopf-identity} below, we prove an analogous identity holds in $\Lambda$.
	Projecting from $\Lambda$ down to $\overline{\Lambda}$ via the (linear) quotient map $s_\lambda\mapsto \overline{s}_{\lambda}$ completes the proof.	
\end{proof}

\begin{lemma}
	\label{lem:hopf-identity}
	For any partitions $\lambda,\mu$ and any partition $\nu\neq \emptyset$,
	\[\sum_{\lambda,\mu}(-1)^{|\lambda|}c_{\lambda,\mu}^{\nu}s_{\lambda'}s_{\mu}=0.\]
\end{lemma}
\begin{proof}
	Recall the Hopf algebra structure on $\Lambda$ over $\mathbb{Z}$. This includes
	\begin{itemize}
		\item the usual multiplication map $\nabla:\Lambda\otimes \Lambda\to \Lambda$ with 
		\[\nabla(s_{\lambda}\otimes s_{\mu}) = s_{\lambda}s_{\mu};\]
		\item the coproduct described previously, with $\Delta:\Lambda\to\Lambda\otimes\Lambda$ with 
		\[\Delta(s_\nu) = \sum_{\lambda,\mu}c_{\lambda,\mu}^{\nu}s_{\lambda}\otimes s_{\nu}; \]
		\item the unit $\eta:\mathbb{Z}\to \Lambda$, the ring homomorphism with $1\mapsto 1$;
		\item the counit $\epsilon:\Lambda\to\mathbb{Z}$ taking $f\in \lambda$ to its constant term $f(0,0,\ldots)$; 
		\item the antipode $S:\Lambda\to\Lambda$ with 
		\[S(s_{\lambda})=(-1)^{|\lambda|}s_{\lambda'}.\]
	\end{itemize}
	The antipode $S$ is characterized by the commutative diagram shown in Figure \ref{fig:diag}.
	\begin{figure}[ht]
		\[
			\begin{tikzcd}[row sep=3.6em,column sep=1em]
				& \Lambda\otimes \Lambda \arrow[rr,"S\otimes\mathrm{id}"] && \Lambda\otimes \Lambda \arrow[dr,"\nabla"] \\
				\Lambda \arrow[ur,"\Delta"] \arrow[rr,"\varepsilon"] \arrow[dr,"\Delta"'] && \mathbb{Z} \arrow[rr,"\eta"] && \Lambda \\
				& \Lambda\otimes \Lambda \arrow[rr,"\mathrm{id}\otimes S"'] && \Lambda\otimes \Lambda \arrow[ur,"\nabla"']
			\end{tikzcd}
		\]
	\caption{The Hopf algebra stucture on $\Lambda$}
	\label{fig:diag}
	\end{figure}
	On a Schur function $s_{\nu}$, we compute
	\begin{align*}
		\eta\circ\epsilon(s_{\nu}) &= \eta(0) = 0,\quad\mbox{and}\\
		\nabla\circ(S\otimes \mathrm{id})\circ\Delta(s_{\nu})&=\nabla\left( \sum_{\lambda,\mu}c_{\lambda,\mu}^{\nu}S(s_{\lambda})\otimes s_{\nu}\right) = \sum_{\lambda,\mu}c_{\lambda,\mu}^{\nu} (-1)^{|\lambda|}s_{\lambda'}s_\mu. \qedhere
	\end{align*}
\end{proof}

\section{Conclusion of the proof} \label{sec:conclusion}

\begin{definition}
	Let $W$ be the subspace 
	\[W= \mathrm{Span}_{\mathbb{Z}}\left\{\linear_{\varphi}(\rho)\mid \rho\in\decomp_{\varphi} \text{ is wide} \right\}\]
\end{definition}

Recall the linear isomorphism $T:V\to \overline{\Lambda}\otimes \overline{\Lambda}$ defined by $T(A_{\lambda,\theta}) = \overline{s}_{\lambda}\otimes \overline{s}_{\theta}$.
\begin{lemma}
	\label{lem:tall-dependent-on-wide}
	Let $\varphi=(n,r,i,j,a,b,N)$ be valid with $N\leq \min(j,n-r-i)$. Then 
	$\linear_{\varphi}(\nu)\in W$
	for any tall $\nu\in \decomp_{\varphi}$.
\end{lemma}
\begin{proof}
	It is enough to show the analogous statement for tensors:
	\[\mathrm{CP}(\overline{s}_{\nu_B}\otimes \overline{s}_{\nu_R})\in \mathrm{Span}_{\mathbb{Z}}\left\{\mathrm{CP}(\overline{s}_{\rho_B}\otimes \overline{s}_{\rho_R})\mid \rho\in\decomp_{\varphi} \text{ is wide} \right\} = T(W).\]
	From $\varphi$ and $\nu$, suppose Algorithm~1 produced $\xi^{(0)},\ldots,\xi^{(K)}=		\mathrm{Reduce}_{\varphi}(\nu)$. Then
	\begin{align*}
		\mathrm{CP}(\overline{s}_{\nu_B}\otimes \overline{s}_{\nu_R})&=\xi^{(0)}\\
		&=\xi^{(K)} -\sum_{m=1}^K (\xi^{(m)}-\xi^{(m-1)})
	\end{align*}
	By Proposition \ref{prop:hopf-identity-application},	
	\[\xi^{(K)}= (-1)^{|\nu_B|+1}(1\otimes \overline{s}_{(\nu_B)'}).\]
	We have $\nu_B'\subseteq (b^a)'=a^b$. Since $a+i\leq n-r$ and $b\leq j$, it follows that $\xi^{(K)}\in T(W)$.
	
	It remains to show that each term $(\xi^{(m)}-\xi^{(m-1)})\in T(W)$. By definition,
	\[\xi^{(m)}-\xi^{(m-1)} = -\gamma\mathrm{CP}(\overline{s}_{\lambda}\otimes (\overline{s}_{\mu^{(1)}}\cdots \overline{s}_{\mu^{(k)}})) \]
	for some $\gamma\in \mathbb{Z}$ and partitions $\lambda,\mu^{(1)},\ldots,\mu^{(k)}$.
	One observes from Algorithm~1 that \[|\lambda|+|\mu^{(1)}|+\cdots+|\mu^{(k)}|=N.\]
	Since we are assuming $N\leq \min(j,n-r-i)$,
	\[ \prod_{p=1}^k \overline{s}_{\mu^{(p)}}\in \mathrm{Span}_\mathbb{Z}\left\{\overline{s}_\theta\mid \theta\subseteq (n-r-i)^j\right\},\]
	say
	\[\prod_{p=1}^k \overline{s}_{\mu^{(p)}} = \sum_{\substack{\theta\subseteq (n-r-i)^j \\|\theta|=N-|\lambda|}} d_{\theta}\overline{s}_{\theta}. \]
	Then if we expand, we obtain
	\begin{align*}
		\xi^{(m)}-\xi^{(m-1)} 
		&=-\gamma\mathrm{CP}(\overline{s}_{\lambda}\otimes (\overline{s}_{\mu^{(1)}}\cdots \overline{s}_{\mu^{(k)}}))\\
		&=-\gamma\mathrm{CP}\biggr(\overline{s}_{\lambda}\otimes \sum_{\substack{\theta\subseteq (n-r-i)^j \\|\theta|=N-|\lambda|}} d_{\theta}\overline{s}_{\theta} \biggr)\\
		&=-\gamma \sum_{\substack{\theta\subseteq (n-r-i)^j \\|\theta|=N-|\lambda|}} d_{\theta}\mathrm{CP}(\overline{s}_{\lambda}\otimes \overline{s}_{\theta}).
	\end{align*}
	Since $\lambda\subseteq \nu_B \subseteq b^a)$, each term $\mathrm{CP}(\overline{s}_{\lambda}\otimes \overline{s}_{\theta})$ is of the form $\mathrm{CP}(\overline{s}_{\rho_B}\otimes \overline{s}_{\rho_R})$ with $\rho\in\decomp_{\varphi}$ wide, it follows that $\xi^{(m)}-\xi^{(m-1)} \in T(W)$.
	\end{proof}

\begin{lemma}
	\label{lem:restricting-equations}
	Suppose $\varphi=(n,r,i,j,a,b,N)$ be valid. Set $\widehat{\varphi}=(\widehat{n},\widehat{r},i,\widehat{j},a,b,N)$ where $\widehat{n}=4^qn$, $\widehat{r}=2^q r$, and $\widehat{j}=2^q j$ for $q\in {\mathbb N}$. Fix any $\nu\in \decomp_{\varphi}$, and let $\widehat{\nu}\in\decomp_{\widehat{\varphi}}$ be the partition with
	\[\widehat{\nu}_B=\nu_B \quad\mbox{and}\quad \widehat{\nu}_R=\nu_R.\]
	Then setting the variables 
	\[\{A_{\lambda,\theta} \mid \lambda\nsubseteq b^a \mbox{ or } \theta\nsubseteq (n-r-i)^j \}\] 
	to zero in $\linear_{\widehat{\varphi}}(\widehat{\nu})$ yields exactly $\linear_{\varphi}(\nu)$.
\end{lemma}
\begin{proof}
	It suffices to note that 
	\[c_{(i^j,\lambda),\theta}^{\nu}= c_{(i^{\widehat{j}}\lambda),\theta}^{\widehat{\nu}}\]
	by Theorem \ref{thm:pictures}, since the skew shapes $\nu/(i^j,\lambda)$ and $\widehat{\nu}/(i^{\widehat{j}}\lambda)$ are (essentially) the same.
\end{proof}

\begin{theorem}
	\label{thm:inconsistent}
	For any valid $\varphi=(n,r,i,j,a,b,N)$, the equations 
	\begin{equation}
		\label{eqn:all-equations}
		\left\{ \linear_{\varphi}(\nu) = d_{\nu}\mid \nu\in \decomp_{\varphi} \right\}
	\end{equation}
	are inconsistent whenever $d_{\nu}=0$ for all wide $\nu$, and $d_{\rho}\neq 0$ for some tall $\rho$.
\end{theorem}
\begin{proof}
	Suppose first that 
	\[N\leq \min(j,n-r-i).\] By Lemma \ref{lem:tall-dependent-on-wide}, we have $\linear_{\varphi}(\rho) \in W$.
	The restriction $d_\nu=0$ when $\nu$ is wide would then force $d_{\rho}=0$, contradicting the hypothesis. Thus the equations (\ref{eqn:all-equations}) are inconsistent.
	
	Now suppose that 
	\[N>\min(j,n-r-i).\] 
	Choose $q$ so that $N\leq \min(2^qj,4^qn-2^qr-i)$. Let 
	\[\widehat{\varphi} = (\widehat{n},\widehat{r},i,\widehat{j},a,b,N),\quad \mbox{where} \quad \widehat{n}=4^qn,\quad \widehat{r}=2^qr,\quad\mbox{and}\quad \widehat{j}=2^qj.\]
	For each $\nu\in \decomp_{\varphi}$, let $\widehat{\nu}\in\decomp_{\widehat{\varphi}}$ be the partition with $\widehat{\nu}_B=\nu_B$ and $\widehat{\nu}_R=\nu_R$. For $\tau\in\decomp_{\widehat{\varphi}}$, define
	\[e_{\tau} = \begin{cases}
		d_\nu &\mbox{ if } \tau=\widehat{\nu} \mbox{ for some } \nu\in\decomp_{\varphi},\\
		0 &\mbox{ otherwise.}
	\end{cases} \]

	By the previous argument, the equations
	\begin{equation}
		\label{eqn:bigger_m}
		\left\{ \linear_{\widehat{\varphi}}(\tau) = e_{\tau}\mid \tau\in \decomp_{\widehat{\varphi}} \right\}
	\end{equation}
	are inconsistent. By Lemma \ref{lem:restricting-equations}, the equations (\ref{eqn:all-equations}) can be obtained by setting certain variables to zero in (\ref{eqn:bigger_m}). Setting variables to zero in an inconsistent linear system produces the same.
\end{proof}

\begin{theorem}
	For any valid $\varphi=(n,r,i,j,a,b,N)$, there is no syzygy of the form 
	\[\sum_{\substack{\lambda\subseteq b^a\\|\lambda|=N}}\overline{s}_{(i^j,\lambda)}f_{\lambda}=\sum_{\substack{\lambda\subseteq b^a\\|\lambda|<N }}f_{\lambda}\overline{s}_{(i^j,\lambda)}\]
	with the left-hand side nonzero. 
\end{theorem}
\begin{proof}
	By Lemma \ref{lem:implication-of-syzygy}, the existence of such a syzygy would contradict Theorem \ref{thm:inconsistent}.
\end{proof}

\begin{corollary}
	\label{cor:no-syzygy}
	For any bigrassmannian $v=v_{r,s,t,n}\in S_n$ and $i,j,a,b$ as in Theorem \ref{thm:main}, the generating set
	\[\{\overline{s}_{\mu} \mid i^j\subseteq \mu \subseteq (i^j,b^a)\}\]
	of $J_v$ in $\overline{\Lambda}$ is minimal.
\end{corollary}

\begin{remark}
	In \cite{RWY}, a transposed version of Theorem \ref{thm:main} is also proved (with minimality conjectured). The generating set is 
	\begin{equation}
		\label{eqn:transposedlist}
		\{\overline{s}_{\mu} \mid i^j\subseteq \mu \subseteq ((i+a)^b,i^{j-b})\}.
	\end{equation}
	These are exactly the conjugates of the generating set (\ref{eqn:minimallist}) whose minimality we prove. One could straightforwardly``transpose'' all arguments in this paper to prove minimality of (\ref{eqn:transposedlist}). Specifically, one would interchange the roles of wide and
	tall equations in Section~\ref{sec:theequations} (with minor modifications) and swap the order of the tensor components in Sections~\ref{sec:Hopf} and \ref{sec:conclusion}.
\end{remark}

\section*{Acknowledgements}
AY thanks Victor Reiner and Alexander Woo for their earlier joint collaboration \cite{RWY}. AY also thanks Alexander
Woo for a recent helpful conversation about the minimality conjecture. AS was partially supported by an NSF postdoctoral fellowship.
AY was partially supported by a Simons Collaboration Grant and an NSF RTG grant in combinatorics.


\begin{thebibliography}{99}

\bibitem{ALP}
E.~Akyildiz, A.~Lascoux, and P.~Pragacz, \emph{Cohomology of Schubert subvarieties of ${\rm GL}_n/P$.} 
J. Differential Geom. 35 (1992), no. 3, 511--519.

\bibitem{Borel}
A.~Borel, \emph{Sur la cohomologie des espaces fibr\'{e}s principaux et des espaces homog\`{e}nes de groupes de Lie compacts.} Ann. of Math. (2) 57 (1953), 115--207.

\bibitem{DMR}
M. Develin, J. Martin, and V. Reiner, \emph{Classification of Ding's Schubert varieties: finer rook equivalence.} 
Canad. J. Math. 59 (2007), no. 1, 36--62.

\bibitem{Ding1}
K.~Ding, \emph{Rook placements and cellular decomposition of partition varieties.} Discrete Math. 170 (1997), no. 1-3, 107--151.

\bibitem{Ding2}
K.~Ding, \emph{Rook placements and classification of partition varieties $B\backslash M_\lambda$.} Commun. Contemp. Math. 3 (2001), no. 4, 495--500.

\bibitem{fultonyt}
W.~Fulton, \emph{Young Tableaux: With Applications to Representation Theory and Geometry}.
London Mathematical Society Student Texts. Cambridge University Press, 1996.

\bibitem{Gasharov.Reiner}
V.~Gasharov and V.~Reiner, \emph{Cohomology of smooth Schubert varieties in partial flag manifolds.} J. London Math. Soc. (2) 66 (2002), no. 3, 550--562. 

\bibitem{James.Peel}
G.~D. James and M.~H. Peel, \emph{Specht series for  skew representations  of  symmetric groups.} 
J. Algebra 56 (1979), 343--364.
 
\bibitem{LS}
 A. Lascoux and M.-P. Sch\"utzenberger, \emph{Treillis et bases des groupes de Coxeter.} 
 Electron. J. Combin. 3 (1996), no. 2, R27, 35 pp. (electronic).
 
\bibitem{Manivel}
L. Manivel, \emph{Symmetric functions, Schubert polynomials and degeneracy loci}. Translated from the
1998 French original by John R. Swallow. SMF/AMS Texts and Monographs, American Mathematical
Society, Providence, 2001.

\bibitem{Reading}
N. Reading, \emph{Order dimension, strong Bruhat order and lattice properties for posets}. 
Order 19 (2002), no. 1, 73--100.

\bibitem{RWY}
V.~Reiner, A.~Woo, and A.~Yong, \emph{Presenting the cohomology of a Schubert variety.} Trans. Amer. Math. Soc. 363 (2011), no. 1, 521--543.

\bibitem{ECII}
R.~P.~Stanley, \emph{Enumerative combinatorics. Vol. 2.} With a foreword by Gian-Carlo Rota and appendix 1 by Sergey Fomin. Cambridge Studies in Advanced Mathematics, 62. Cambridge University Press, Cambridge, 1999. xii+581 pp.

\bibitem{Zelev}
A. V. Zelevinsky, \emph{A generalization of the Littlewood-Richardson rule and the
Robinson-Schensted-Knuth correspondence}, J. Algebra 69 (1981), 82--94.

\end{thebibliography}
\end{document}